\documentclass[12pt, a4paper]{amsart}

\usepackage{amsfonts,amsmath,amssymb,amscd}  \usepackage[all]{xy}

\def\N{{\Bbb N}}
\def\Z{{\Bbb Z}}

\def\F{{\Bbb F}}
\def\G{{\Bbb G}}
 
\def\Im{{\mathrm{Im}}}

\def\Ker{{\mathrm{Ker}}}

\def\Im{\mathrm{Im}}

\def\Spec{{\mathrm{Spec}}}

\newtheorem{theorem}{Theorem}[section]

\newtheorem{prop}[theorem]{Proposition}
\newtheorem{corollary}[theorem]{Corollary}

\theoremstyle{definition}
\newtheorem{definition}[theorem]{Definition}
\newtheorem{rem}[theorem]{Remark}

\newtheorem{notation}[theorem]{Notation}
\newtheorem{example}[theorem]{Example}
\usepackage{color}

\newcommand\pf{\begin{proof}}
\newcommand\epf{\end{proof}}

\usepackage{mathrsfs}

\title[]
{On the unit group scheme of the group algebra of a certain non-commutative finite flat group scheme over an $\Bbb{F}_{p}$-algebra\\
}

\author[Yuji~Tsuno]{Yuji~Tsuno}
\address{Yuji Tsuno: National Institute of Technology, Wakayama College,
77 Noshima, Nada-cho, Gobo, Wakayama, Japan 644-0023}
\email{tsuno@wakayama-nct.ac.jp}

\begin{document} 

\begin{abstract}
Suwa investigated the unit group scheme of the group ring associated with a finite flat group scheme and provided a characterization of torsors possessing the normal base property for such schemes. In this paper, we examine the unit group scheme of the group ring for a specific non-commutative finite flat group scheme and characterize torsors with the normal base property in this context. Moreover, in connection with the Noether problem for Hopf algebras proposed by Kassel and Masuoka, we compute the quotient of the unit group scheme under the action of this non-commutative finite flat group scheme.
\end{abstract}

\maketitle

\noindent

\medskip

\section{Introduction}\label{sec:intro}
\renewcommand{\thefootnote}{} 
\footnote[0]{\sc Mathematics Subject Classification (2020). Primaly 13B05; Secondary 14L15, 12G05\\
 \ \ \  Keywords and phrases. Finite flat group scheme, normal basis problem, Hopf-Galois extension, cleft extension
}
\renewcommand{\thefootnote}{\arabic{footnote}} 

An elementary proof of Kummer theory using Lagrange resolvents is well known. Building on the normal basis theorem in field theory, Serre \cite{Se} reformulated this approach as follows:

Let $\varGamma$ be a finite group, $k$ a field, and $U(\varGamma_{k})$ the algebraic group representing the unit group of the group algebra $k[\varGamma]$. Then any Galois extension $K/k$ with Galois group $\varGamma$ can be obtained from the Cartesian diagram 
\[\begin{CD}
 \Spec \ K  @>>> U(\varGamma_{k}) \\
 @VVV  @VVV \\
 \Spec \ k @>>> U(\varGamma_{k})/\varGamma\ .
\end{CD}\]

\vspace{3mm}

Moreover, Serre presented another proof of both Kummer theory and Artin-Schreier-Witt theory by constructing the following short exact sequences:

\[\begin{CD}
 0 @>>>\boldsymbol{\mu_{n,k}} @>>> U(\boldsymbol{\mu_{n,k}}) @>>> U(\boldsymbol{\mu_{n,k}})/\boldsymbol{\mu_{n,k}} @>>> 0 \\
 @.  @\vert  @VVV  @VVV  @. \\
 0 @>>> \boldsymbol{\mu_{n,k}} @>>> \G_{m,k} @>n>> \G_{m,k} @>>> 0
\end{CD}\]

($k$ contains all the $n$-th roots of unity and $n$ is invertible in $k$)
and
\[\begin{CD}
 0 @>>>  \Z/p^{n}\Z @>>> U(\Z/p^{n}\Z) @>>> U(\Z/p^{n}\Z)_{k}/\Z/p^{n}\Z@>>> 0 \\
 @.  @\vert  @VVV  @VVV  @. \\
 0 @>>>  \Z/p^{n}\Z @>>> W_{n,k} @>{F-1}>> W_{n,k} @>>> 0
\end{CD}\]

($k$ is of characteristic $p$).

Subsequently, Suwa \cite{Su1}, \cite{Su2} reformulated Serre's method as the sculpture problem and reconstructed it by incorporating the embedding problem. Furthermore, Suwa's reformulation of Serre's method can be extended to finite flat group schemes. To achieve such an extension, it was necessary to reinterpret the concept of Hopf-Galois extensions with the normal basis property (i.e., cleft extensions) within the framework of algebraic geometry, as follows:

\vspace{3mm}

\begin{definition}[{\cite[Definition 2.9]{Tsu2}}]
Let $S$ be a scheme, $\varGamma$ an affine group $S$-scheme and $X$ a right $\varGamma$-torsor over $S$. We say that a right $\varGamma$-torsor $X$ is cleft if there exists an isomorphism of ${\mathcal O}_S$-modules $\varphi:{\mathcal O}_{\varGamma}\overset{\sim}\rightarrow {\mathcal O}_X$ such that the following diagram commutes:
\[\begin{CD}
 {\mathcal O}_{\varGamma} @>{\varphi}>> {\mathcal O}_X \\
 @VV\Delta V  @VV\rho V \\
 {\mathcal O}_{\varGamma}\otimes_{{\mathcal O}_S} {\mathcal O}_{\varGamma} @>{\varphi \otimes Id}>> {\mathcal O}_X \otimes_{{\mathcal O}_S} {\mathcal O}_{\varGamma}\ .
\end{CD}\]
is commutative. Here $\Delta$ denotes the comultiplication of ${\mathcal O}_S$-Hopf algebra ${\mathcal O}_{\varGamma}$ and $\rho$ the right ${\mathcal O}_{\varGamma}$-comodule algebra structure homomorphism of ${\mathcal O}_S$-algebra ${\mathcal O}_X$. 
\end{definition}

Let $S$ be a scheme and $\varGamma$ an affine $S$-group scheme such that ${\mathcal O}_{\varGamma}$ is a locally free ${\mathcal O}_{S}$-module of finite rank. We can then consider the unit group scheme $U(\varGamma)$ associated with the group algebra of $\varGamma$. Moreover, via the canonical closed embedding  $i: \varGamma \rightarrow U(\varGamma)$, we obtain an exact sequence of schemes over $S$ with values in pointed sets:
\[1 \longrightarrow \varGamma \overset{i}\longrightarrow U(\varGamma) \longrightarrow U(\varGamma)/\varGamma \longrightarrow 1.\]
If $\varGamma$ is commutative, this exact sequence is called Grothendieck resolution (cf. [6. Sec 6]).
Furthermore, Suwa showed that $U(\varGamma)$ is a right $\varGamma$-cleft torsor over $U(\varGamma)/\varGamma$ and provided the following characterization of cleft torsors under a finite flat group scheme:

\vspace{3mm}
\begin{theorem}[{\cite{Su3}, \cite{Tsu2}}]
Let $S$ be a scheme and $\varGamma$ an affine $S$-group scheme such that ${\mathcal O}_{\varGamma}$ is a locally free ${\mathcal O}_{S}$-module of finite rank. Then, a $\varGamma$-torsor $X$ over $S$ is cleft if and only if there exists a Cartesian diagram

\[\begin{CD}
 X  @>>> U(\varGamma)\\
 @VVV  @VVV \\
 S @>>> U(\varGamma)/\varGamma\ .
\end{CD}\]

\end{theorem}
This theorem was established by the author of the present paper for the case where $\varGamma$ is commutative, and by Suwa for the general case where $\varGamma$ is not necessarily commutative.

\vspace{3mm}

With the notation of the above theorem, we obtain the following results.

\vspace{3mm}
Let $G$ be a flat affine group scheme over $S$.

 (A) Assume that $e: \varGamma \rightarrow G$ is a closed embedding of group schemes and that there exists a commutative diagram

 \[\begin{CD}
 \varGamma @>i>> U(\varGamma) \\
 @VVV  @VVV \\
 \varGamma @>e>> G.
\end{CD}\]
If a $\varGamma$-torsor over $S$ is cleft, then there exists a morphism $X \rightarrow G$ and $S \rightarrow G/\varGamma$ such that the diagram
\[\begin{CD}
 X @>>> G \\
 @VVV  @VVV \\
 S @>>> G/\varGamma
\end{CD}\]
is cartesian.

(B) Assume that $e: \varGamma \rightarrow G$ is a closed embedding of group schemes and that there exists a commutative diagram
 
 \[\begin{CD}
 \varGamma @>e>> G \\
 @VVV  @VVV \\
 \varGamma @>i>> U(\varGamma).
\end{CD}\]
Then, if a $\varGamma$-torsor X over $S$ is defined by a cartesian diagram
\[\begin{CD}
 X @>>> G \\
 @VVV  @VVV \\
 S @>>> G/\varGamma
\end{CD}\]
then $X$ is a cleft $\varGamma$-torsor.

We refer to problems
(A) and (B) as the sculpture problem and the embedding problem, respectively.

From the above discussion, we obtain the following corollary:

\vspace{3mm}
\begin{corollary} Under the notation of Theorem 1.2, let $G$ be a flat affine group scheme over $S$. Suppose there exist commutative diagrams
\[\begin{CD}
 \varGamma @>i>> U(\varGamma) \\
 @VV{\wr}V  @VVV \\
 \varGamma @>>> G
\end{CD}\] 
and
\[\begin{CD}
 \varGamma @>>> G \\
 @VV{\wr}V  @VVV \\
 \varGamma @>i>> U(\varGamma)\ ,
\end{CD}\]
where $\varGamma \rightarrow G$ is a closed embedding of group schemes and $i: \varGamma \rightarrow U(\varGamma)$ is the canonical closed embedding of group schemes. Then, a $\varGamma$-torsor $X$ over $S$ is cleft if and only if $X$ is defined by the Cartesian diagram
\[\begin{CD}
 X @>>> G \\
 @VVV  @VVV \\
 S @>>> G/\varGamma\ .
\end{CD}\]
\end{corollary}

In this paper, building on the above framework, we investigate the case of a certain non-commutative finite flat group scheme. We describe the unit group scheme of its group algebra and analyze both the sculpture problem and the embedding problem.

Specifically, we focus on the following non-commutative group schemes:

\begin{definition} (=Definition 2.3) Let $p$ be a prime number, $R$ an $\F_{p}$-algebra and $\lambda \in R$.
The non-commutative group scheme ${\mathcal H}^{(\lambda)}_{R}$ is defined by 

\[{\mathcal H}^{(\lambda)}_{R}=\Spec\,R[X,Y, \frac{1}{1+\lambda X}]\]

with

(a) the comultiplication map:
\[X \mapsto X\otimes 1+(1+\lambda X)\otimes X, \ \ \ Y \mapsto Y\otimes 1+(1+\lambda X)\otimes Y,\]

(b) the counit map:
\[X\mapsto 0, \ \ \  Y\mapsto 0,\]

(c) the antipode:
\[X \mapsto \frac{-X}{1+\lambda X}, \ \ \ Y\mapsto \frac{-Y}{1+\lambda X}.\] 

\end{definition}

We define the Frobenius map $F: {\mathcal H}^{(\lambda)}_{R} \rightarrow {\mathcal H}^{(\lambda^{p})}_{R}$ defined by
\[X \mapsto X^{p}, Y \mapsto Y^{p}: R[X,Y, \frac{1}{1+\lambda^{p} X}] \rightarrow R[X,Y, \frac{1}{1+ \lambda X}].\]
Put $G^{(\lambda)}_{R}=\Ker[F:{\mathcal H}^{(\lambda)}_{R} \rightarrow {\mathcal H}^{(\lambda^{p})}_{R}]$. 
Then $G^{(\lambda)}_{R}$ is a non-commutative finite flat group scheme.

For this group scheme, we obtain the following results:

\begin{theorem}Under the notation of Definition 1.4, there exist commutative diagrams

\[\begin{CD}
 G^{(\lambda)}_{R} @>i>> U(G^{(\lambda)}_{R}) \\
 @\vert  @VV\tilde{\chi}V \\
 G^{(\lambda)}_{R} @>e>> {\mathcal H}^{(\lambda)}_{R}.
\end{CD} \ \ \  \begin{CD}
 G^{(\lambda)}_{R} @>e>> {\mathcal H}^{(\lambda)}_{R} \\
 @\vert  @VV\tilde{\sigma}V \\
 G^{(\lambda)}_{R} @>i>> U(G^{(\lambda)}_{R}).
\end{CD}\]

\end{theorem}

Additionally, the quotient scheme $U(G^{(\lambda)}_{R}) / G^{(\lambda)}_{R}$ is described in Theorem 4.4. Kassel and Masuoka \cite{KM} studied the Noether problem for Hopf algebras, and Theorem 4.4 offers a significant positive instance within this context. 
We conclude the article by presenting an example of non-cleft $G^{(\lambda)}_{R}$-torsors.

\vspace{5mm}

\begin{notation} Throughout this article, $p$ denotes a prime number. For a scheme $S$ and a group scheme $\varGamma$ over $S$, $H^{1}(S, \varGamma)$ denotes the set of isomorphism classes of right $\varGamma$-torsors over $S$. (For further details, see Demazure-Gabriel [1, Ch III. 4].)

\end{notation} 

\vspace{2mm}

\noindent{\bf List of group schemes}

$\G_{a,R}$: recalled in 2.1,

$\G_{m,R}$: recalled in 2.1,

${\mathcal G}^{(\lambda)}_{R}: \text{recalled in} \ \ 2.2$,

${\mathcal H}^{(\lambda)}_{R}$: defined in 2.3,

$G^{(\lambda)}_{R}$: defined in 2.4,

$U(\varGamma)$: recalled in subsection 3.1.

\vspace{5mm}

\section{Preliminary}\label{sec:prelimi}

\vspace{2mm}
\begin{definition} Let $R$ be a commutative ring. The additive  group scheme $\G_{a,R}$ over $R$ is defined by \[\G_{a,R}=\Spec\,R[T]\]
with

(a) the multiplication: $T\mapsto T\otimes1+1\otimes T$,

(b) the unit: $T\mapsto 0$,

(c) the inverse: $T\mapsto -T$.

\vspace{2mm}
On the other hand, the multiplicative group scheme $\G_{m,R}$ over $R$ is defined by \[\G_{m,R}=\Spec\,R[T,\frac{1}{T}]\]
with

(a) the multiplication: $T\mapsto T\otimes T$,

(b) the unit: $T\mapsto 1$,

(c) the inverse: $T\mapsto {1}/{T}$.

\end{definition}
\vspace{2mm}

\begin{definition} Let $R$ be a commutative ring and $\lambda\in R$. The commutative group scheme ${\mathcal G}^{(\lambda)}_{R}$ over $R$ is defined by
\[{\mathcal G}^{(\lambda)}_{R}=\Spec\,R[T,\frac{1}{1+\lambda T}]\]
with
(a) the comultiplication map: $T\mapsto T\otimes1+1\otimes T+\lambda T\otimes T$,

(b) the counit: $T\mapsto 0$,

(c) the antipode: $T\mapsto -T/(1+\lambda T)$.

A homomorphism $\alpha^{(\lambda)}:{\mathcal G}^{(\lambda)}_{R}\rightarrow\G_{m,R}$ of group schemes over $R$ is defined by
\[U\mapsto \lambda T+1:R[U,\frac{1}{U}]\longrightarrow R[T,\frac{1}{1+\lambda T}].\]
If $\lambda$ is invertible in $R$, then $\alpha^{(\lambda)}$ is an isomorphism.
Conversely, if $\lambda=0$, the scheme ${\mathcal {G}}^{(\lambda)}_{R}$ reduces to the additive group scheme $\G_{a,R}$.
\end{definition}

\begin{definition} Let $R$ be a commutative ring and $\lambda \in R$.
The non-commutative group scheme ${\mathcal H}^{(\lambda)}_{R}$ is defined by 

\[{\mathcal H}^{(\lambda)}_{R}=\Spec\,R[X,Y, \frac{1}{1+\lambda X}]\]

with structure maps:

(a) comultiplication:
\[X \mapsto X\otimes 1+(1+\lambda X)\otimes X, \ \ \ Y \mapsto Y\otimes 1+(1+\lambda X)\otimes Y,\]

(b) counit:
\[X\mapsto 0, \ \ \  Y\mapsto 0,\]

(c) antipode:
\[X \mapsto \frac{-X}{1+\lambda X}, \ \ \ Y\mapsto \frac{-Y}{1+\lambda X}.\] 

${\mathcal H}_{R}^{(\lambda)}$ is an extension of ${\mathcal G}^{(\lambda)}_{R}$ by $\G_{a,R}$. Indeed, we define a group homomorphism $i: \G_{a,R} \rightarrow {\mathcal H}^{(\lambda)}_{R}$
by \[R[X,Y, \frac{1}{1+\lambda X}] \rightarrow R[T]: X \mapsto 0, Y \mapsto T\]
Also, we define a group homomorphism $epi: {\mathcal H}^{(\lambda)}_{R} \rightarrow {\mathcal G}^{(\lambda)}_{R}$ by
\[R[X,Y, \frac{1}{1+\lambda X}] \rightarrow R[T]: T \rightarrow X\].

We obtain an exact sequence of group schemes
\[0 \longrightarrow \G_{a,R} \overset{i}\longrightarrow {\mathcal H}^{(\lambda)}_{R} \overset{epi}\longrightarrow {\mathcal G}^{(\lambda)}_{R}\longrightarrow 0\]
by these morphisms. 

\end{definition}

\begin{definition} Let $p$ be a prime number and $R$ an $\F_{p}$-algebra. Then we define the Frobenius map $F: {\mathcal H}^{(\lambda)}_{R} \rightarrow {\mathcal H}^{(\lambda^{p})}_{R}$ defined by
\[X \mapsto X^{p}, Y \mapsto Y^{p}: R[X,Y, \frac{1}{1+\lambda^{p} X}] \rightarrow R[X,Y, \frac{1}{1+ \lambda X}].\]
Put $G^{(\lambda)}_{R}=\Ker[F:{\mathcal H}^{(\lambda)}_{R} \rightarrow {\mathcal H}^{(\lambda^{p})}_{R}]$. 
Then $G^{(\lambda)}_{R}$ is defined by
\[G^{(\lambda)}_{R}=\Spec\,R[X,Y]/(X^{p}, Y^{p})\]
with

(a) comultiplication:
\[X \mapsto X\otimes 1+(1+\lambda X)\otimes X, \ \ \ Y \mapsto Y\otimes 1+(1+\lambda X)\otimes Y,\]

(b) counit:
\[X\mapsto 0, \ \ \  Y\mapsto 0,\]

(c) antipode:
\[X \mapsto \frac{-X}{1+\lambda X}, \ \ \ Y\mapsto \frac{-Y}{1+\lambda X}.\] 

\vspace{5mm}
In this paper, the $R$-Hopf algebra $R[X,Y]/(X^{p}, Y^{p})$ representing $G^{(\lambda)}_{R}$ is denoted $A^{(\lambda)}_{R}$.

\end{definition}

\vspace{5mm}

\section{Main result}\label{sec:main}

\vspace{3mm}

\subsection{$U(\varGamma)$ for a finite flat group scheme $\varGamma$.}
We recall the group algebra scheme $A(\varGamma)$ and its unit group scheme $U(\varGamma)$ for a finite flat group scheme $\varGamma$. For details of these group schemes, we refer to [11, Section 2].

Let $S$ be a scheme and $\varGamma$ an affine group scheme over $S$. 
Let $A(\varGamma)$ represent the ring functor defined by $T \mapsto \mathit{Hom}_{{\mathcal{O}}_{S}}({\mathcal{O}}_{\varGamma}, {\mathcal{O}}_{T})$ for an affine $S$-scheme $T$, where multiplication in $\mathit{Hom}_{{\mathcal{O}}_{S}}({\mathcal{O}}_{\varGamma}, {\mathcal{O}}_{T})$ is given by the convolution product. Then, $A(\varGamma)$ is an $S$-ring scheme. Moreover, we define a functor $U(\varGamma)$ by $U(\varGamma)(T)=A(\varGamma)(T)^{\times}$ for an affine $S$-scheme $T$. Then, $U(\varGamma)$ is a sheaf of groups for the fppf-topology over $S$. If ${\mathcal{O}}_{\varGamma}$ is a locally free ${\mathcal{O}}_{S}$-module of finite rank, $U(\varGamma)$ is represented by an affine smooth group scheme over $S$.
Let $R$ be a commutative ring. We assume that $S=\Spec\,R$ and $\varGamma=\Spec\,H$, where $H$ is a free $R$-module of finite rank. We take a basis $\{e_{1}, \dots, e_{n}\}$ of $H$ over $R$. Let $S_{R}(H)$ denote the symmetric $R$-algebra associated with the $R$-module $H$. For each $i$, let $T_{e_{i}}$ denote the image of $e_{i}$ by the canonical injection $H \rightarrow S_{R}(H)$. Moreover, we define a linear combination $R_{ij}(e_{1}, \dots, e_{n})=\displaystyle \sum_{k=1}^{n}c_{ijk}e_{k}$ for each $1 \leq i, j \leq n$ by
\[\varDelta_{H}(e_{j})=\displaystyle \sum_{i=1}^{n}e_{i} \otimes R_{ij}(e_{1}, \dots , e_{n}).\]
Then, we obtain that $A(\varGamma)=\Spec\,S_{R}(H)=\Spec\,R[T_{e_{1}}, \dots T_{e_{n}}]$ with
the comultiplication map
\[\varDelta(T_{e_{j}})=\displaystyle \sum_{i=1}^{n}T_{e_{i}} \otimes R_{ij}(T_{e_{1}}, \dots, T_{e_{n}}),\]
where $R_{ij}(T_{e_{1}}, \dots, T_{e_{n}})=\displaystyle \sum_{k=1}^{n}c_{ijk}T_{e_{k}}$ and the counit map $\varepsilon(T_{e_{j}})=\varepsilon_{H}(e_{j})$. Moreover, let $A$ be an $R$-algebra. Then, the multiplication of $A(\varGamma)(A)$ is defined by
\[(a_{1}, a_{2}, \dots, a_{n})(b_{1}, b_{2}, \dots, b_{n})\]
\[=(\displaystyle \sum_{j=1}^{n}R_{1j}(a_{1},a_{2}, \dots, a_{n})b_{j},  \sum_{j=1}^{n}R_{2j}(a_{1},a_{2}, \dots, a_{n})b_{j}, \dots,\sum_{j=1}^{n}R_{nj}(a_{1},a_{2}, \dots, a_{n})b_{j} )\]
Hence, $(a_{1}, a_{2}, \dots, a_{n}) \in A(\varGamma)(A)$ is invertible if and only if $\det(R_{ij}(a_{1}, a_{2}, \dots, a_{n}))$ is invertible in $A$.
Therefore, we have 
\[U(\varGamma)=\Spec\,R[T_{e_{1}}, T_{e_{2}}, \dots T_{e_{n}}, \frac{1}{D}],\]
where $D=\det(R_{ij}(T_{e_{1}}, T_{e_{2}}, \dots, T_{e_{n}}))$.

Moreover, the $R$-homomorphism $i^{\#}:  R[T_{e_{1}}, T_{e_{2}}, \dots T_{e_{n}}, {1}/{D}] \rightarrow H$ defined by $T_{e_{i}} \mapsto e_{i}$ induces a closed immersion of group schemes
$i: \varGamma \rightarrow U(\varGamma)$.  

The above construction outlines the proof of [11, Theorem 2.6]. If $R$ is a field, the Hopf algebra $R[T_{1},T_{2}, \dots T_{n}, {1}/{D}]$ coincides the commutative free Hopf algebra generated by $H$ constructed by Takeuchi [13]. Thus [11] provides an algebraic geometric interpretation of Takeuchi's result. 

\vspace{5mm}

\subsection{A description of $U(G^{(\lambda)}_{R})$.}

We now describe the unit group scheme $U(G^{(\lambda)}_{R})$ of the group algebra associated  the finite flat group scheme $G^{(\lambda)}_{R}$. 

\vspace{3mm}
\begin{prop} Let $p$ be a prime number, $R$ an $\F_{p}$-algebra and $\lambda \in R$. Put $G^{(\lambda)}_{R}=\Ker[F:{\mathcal H}^{(\lambda)}_{R} \rightarrow {\mathcal H}^{(\lambda^{p})}_{R}]$. Then $A(G^{(\lambda)}_{R})$ is defined by
\[A(G^{(\lambda)}_{R})=\Spec\,R[T_{X^{r_{1}}Y^{r_{2}}}]_{1 \leq r_{1}, r_{2} \leq p-1}\]
with

(a) the comultiplication map:
\[T_{X^{r_{1}}Y^{r_{2}}} \mapsto \sum_{k=0}^{r_{1}}\sum_{k'=0}^{k}\sum_{l=0}^{r_{2}}\sum_{l'=0}^{l} {\binom{k}{k'}} {\binom{r_{1}}{k}} {\binom{r_{2}}{l}} {\binom{l}{l'}}\lambda^{k'+l'}T_{X^{r_{1}-k+k'+l'}Y^{r_{2}-l}}\otimes T_{X^{k}Y^{l}} \ \ \ (r_{1}+r_{2} < p)\]
\[T_{X^{r_{1}}Y^{r_{2}}} \mapsto \sum_{k=0}^{r_{1}}\sum_{k'=0}^{k}\sum_{l=0}^{r_{2}}\sum_{l'=0}^{l} C(r_{1},r_{2},k,l,k',l') \lambda^{k'+l'}T_{X^{r_{1}-k+k'+l'}Y^{r_{2}-l}}\otimes T_{X^{k}Y^{l}} \ \ \ (r_{1}+r_{2} \geq p),\]
where \[C(r_{1},r_{2},k,l,k',l')=
\begin{cases}
0 & (r_{1}-k+k'+l'\geq p)\\
{\binom{k}{k'}} {\binom{r_{1}}{k}} {\binom{r_{2}}{l}} {\binom{l}{l'}} & (r_{1}-k+k'+l'<p).\\ 
\end{cases}\]

(b) the counit map:
\[T_{1} \mapsto 1, \ \ \ T_{X^{r_{1}}Y^{r_{2}}}\mapsto 0, (r_{1} \neq 0, \ \text{or} \ r_{2} \neq 0).\]

\end{prop}

\vspace{3mm}

\begin{proof} Note that 
\[\beta^{(\lambda)}=\{X^{r_{1}}Y^{r_{2}} | 0\leq r_{1} \leq p-1, 0 \leq r_{2} \leq p-1\}\]
is a basis of $R$-module $A^{(\lambda)}_{R}$.
Here, we have
\begin{align*}
\Delta(X^{r_{1}})&=(X \otimes 1+(1+\lambda X) \otimes X)^{r_{1}}=\sum_{k=0}^{r_{1}} {\binom{r_{1}}{k}} X^{r_{1}-k}(1+\lambda X)^{k} \otimes X^{k}\\
             &=\sum_{k=0}^{r_{1}} {\binom{r_{1}}{k}} X^{r_{1}-k}(\sum _{k'=0}^{k} {\binom{k}{k'}}\lambda^{k'}X^{k'})\otimes X^{k}=\sum_{k=0}^{n}\sum_{k'=0}^{k} {\binom{k}{k'}} {\binom{r_{1}}{k}}\lambda^{k'}X^{r_{1}-k+k'}\otimes X^{k}\\
\Delta(Y^{r_{2}})&=(Y \otimes 1 + (1+\lambda X)\otimes Y)^{r_{2}}=\displaystyle \sum_{k=0}^{r_{2}}(Y^{r_{2}-k}\otimes 1)(1+\lambda X)^{k}\otimes Y^{k}\\
&=\displaystyle\sum_{k=0}^{r_{1}}\displaystyle \sum_{k'=0}^{k} {\binom{r_{1}}{k}} {\binom{k}{k'}}\lambda^{k'}X^{k'}Y^{r_{1}-k}\otimes Y^{k}
\end{align*}

Hence, if $r_{1}+r_{2}<p$, \[\Delta(X^{r_{1}}Y^{r_{2}})=\sum_{k=0}^{r_{1}}\sum_{k'=0}^{k}\sum_{l=0}^{r_{2}}\sum_{l'=0}^{l} {\binom{k}{k'}} {\binom{r_{1}}{k}} {\binom{r_{2}}{l}} {\binom{l}{l'}}\lambda^{k'+l'}X^{n-k+k'+l'}Y^{r_{2}-l}\otimes X^{k}Y^{l}.\]
If $r_{1}+r_{2} \geq p$,
\[\Delta(X^{r_{1}}Y^{r_{2}})=\sum_{k=0}^{r_{1}}\sum_{k'=0}^{k}\sum_{l=0}^{m}\sum_{l'=0}^{l} C(r_{1},r_{2},k,l,k',l')\lambda^{k'+l'}X^{r_{1}-k+k'+l'}Y^{r_{2}-l}\otimes X^{k}Y^{l},\]
where \[C(r_{1},r_{2},k,l,k',l')=
\begin{cases}
0 & (r_{1}-k+k'+l'\geq p)\\
{\binom{k}{k'}} {\binom{r_{1}}{k}} {\binom{r_{2}}{l}} {\binom{l}{l'}} & (r_{1}-k+k'+l'<p).\\ 
\end{cases}\]
Setting $e_{1+r_{1}+pr_{2}}=X^{r_{1}}Y^{r_{2}}$ for $0 \leq r_{1} \leq p-1$,  $0 \leq r_{2} \leq p-1$, we obtain from Subsection 3.1 the multiplication and unit of $A(G^{(\lambda)}_{R})$.
\end{proof}
\vspace{3mm}

\begin{corollary} Let $p$ be a prime number, $R$ an $\F_{p}$-algebra and $\lambda \in R$. Put $G^{(\lambda)}_{R}=\Ker[F:{\mathcal H}^{(\lambda)}_{R} \rightarrow {\mathcal H}^{(\lambda^{p})}_{R}]$. Then the unit group scheme $U(G^{(\lambda)}_{R})$ of the group algebra, $G^{(\lambda)}_{R}$ is given by
\[U(G^{(\lambda)}_{R})=\Spec\,R[T_{X^{r_{1}}Y^{r_{2}}}, \frac{1}{D}]_{1 \leq r_{1}, r_{2} \leq p-1},\]
where $D=\displaystyle\prod_{r=0}^{p-1}(\displaystyle\sum_{k=0}^{r}{\binom{r}{k}}\lambda^{k}T_{X^{k}})^{p}$. 
\end{corollary}
\vspace{3mm}

\begin{proof} Set $e_{1+r_{1}+pr_{2}}=X^{r_{1}}Y^{r_{2}}$ for $0 \leq r_{1} \leq p-1$,  $0 \leq r_{2} \leq p-1$. Then the right regular representation $(R_{ij})_{1 \leq i,j \leq p^{2}}$ of $A^{(\lambda)}_{R}$ with respect to the basis $\beta^{(\lambda)}$ is given by 
 
\[R_{ij}(e_{1}, \dots, e_{p^{2}})=\begin{cases}
0 & (i>j) \\[2mm]
\displaystyle\sum_{k=0}^{r}{\binom{r}{k}}\Lambda^{k}X^{k} & (i=j)\\
\text{a polynomial of $X$ and $Y$} & (i<j),
\end{cases}\]
where $r$ is the remainder of $i-1$ modulo $p$. 

Therefore, we obtain that \[D=\det(R_{ij}(T_{X^{0}}, T_{X^{1}}, \dots ,T_{X^{p-1}Y^{p-1}}))=\displaystyle\prod_{r=0}^{p-1}(\displaystyle\sum_{k=0}^{r}{_{r}C_{k}}\lambda^{k}T_{X^{k}})^{p}.\]
because the matrix $(R_{ij})_{1 \leq i,j \leq p^{2}}$ is an upper triangular matrix.
\end{proof}

\vspace{5mm}
In this paper, the $R$-Hopf algebra $R[T_{X^{r_{1}}Y^{r_{2}}}, \frac{1}{D}]_{1 \leq r_{1}, r_{2} \leq p-1}$ representing $U(\varGamma^{(\lambda)}_{R})$ is denoted by $S(A_{R}^{(\lambda)})_{\Theta}$. 

\vspace{5mm}
\subsection{The sculpture problem and the embedding problem for $G^{(\lambda)}_{R}$.}

\begin{theorem} Let $R$ be an $\F_{p}$-algebra and $\lambda \in R$. Put $G^{(\lambda)}_{R}=\Ker[F:{\mathcal{H}}^{(\lambda)}_{R} \rightarrow {\mathcal{H}}^{(\lambda^{p})}_{R}]$. Then

(1) a morphism of group schemes 
\[\tilde{\chi}: U(G^{(\lambda)}_{R})=\Spec\,R[T_{X^{r_{1}}Y^{r_{2}}}, \frac{1}{D}]_{1 \leq r_{1}, r_{2} \leq p-1} \rightarrow {\mathcal{H}}^{(\lambda)}_{R}=\Spec\,R[X,Y, \frac{1}{1+\lambda X}]\]
is defined by  
\[X \mapsto \frac{T_{X}}{T_{1}}, \ Y \mapsto T_{Y}.\]
Moreover, a diagram of group schemes
\[\begin{CD}
 G^{(\lambda)}_{R} @>{i}>> U(G^{(\lambda)}_{R}) \\
 @\vert  @VV{\tilde{\chi}}V \\
 G^{(\lambda)}_{R} @>>{e}> {\mathcal H}^{(\lambda)}_{R}
\end{CD}\]
is commutative.

(2) a morphism of group schemes

\[\tilde{\sigma}: {\mathcal {H}}^{(\lambda)}_{R}=\Spec\,R[X,Y, \frac{1}{1+\lambda X}] \rightarrow U(G^{(\lambda)}_{R})=\Spec\,R[T_{X^{r_{1}}Y^{r_{2}}}, \frac{1}{D}]_{1 \leq r_{1}, r_{2} \leq p-1}\]
is defined by 

\[T_{X^{r_{1}}Y^{r_{2}}} \mapsto X^{r_{1}}Y^{r_{2}}. \]  Moreover, a diagram of group schemes
\[\begin{CD}
 G^{(\lambda)}_{R} @>{i}>> U(G^{(\lambda)}_{R}) \\
 @\vert  @AA{\tilde{\sigma}}A \\
 G^{(\lambda)}_{R} @>>{e}> {\mathcal {H}}^{(\lambda)}_{R}
\end{CD}\]
is commutative.
Here, \[i:G^{(\lambda)}_{R}=\Spec\,R[X,Y]/(X^{p}, Y^{p}) \rightarrow U(G^{(\lambda)}_{R})=\Spec\,R[T_{X^{r_{1}}Y^{r_{2}}}, \frac{1}{D}]_{1 \leq r_{1}, r_{2} \leq p-1}\]
is the embedding defined by $T_{X^{r_{1}}Y^{r_{2}}} \mapsto X^{r_{1}}Y^{r_{2}}$. 
\[e:G^{(\lambda)}_{R}=\Spec\,R[X,Y]/(X^{p}, Y^{p}) \rightarrow {\mathcal {H}}^{(\lambda)}_{R}=\Spec\,R[X,Y, \frac{1}{1+\lambda X}]\]
is the embedding defined by $X \mapsto X, Y \mapsto Y$.

\end{theorem}
\vspace{3mm}
\begin{proof} (1) A homomorphism of $R$-Hopf algebras $\tilde{\chi}^{\#}:R[X,Y, 1/1+\lambda X] \rightarrow R[T_{X^{r_{1}}Y^{r_{2}}}, 1/D]_{1 \leq r_{1}, r_{2} \leq p-1}$
defined by \[X \mapsto \frac{T_{X}}{T_{1}}, \ Y \mapsto T_{Y}\] is well-defined because $T_{1}+\lambda T_{X}$ is invertible in $R[T_{X^{r_{1}}Y^{r_{2}}}, 1/D]_{1 \leq r_{1}, r_{2} \leq p-1}$. Moreover, we have
\[\Delta_{{\mathcal{H}}^{(\lambda)}}(X)=X \otimes 1 + (1+\lambda X) \otimes X,\]
\[\Delta_{{\mathcal{H}}^{(\lambda)}}(Y)=Y \otimes 1 + (1+\lambda X) \otimes Y.\]
where $\Delta_{{\mathcal{H}}^{(\lambda)}}$ denotes the comultiplication map of the coordinate ring of ${\mathcal{H}}^{(\lambda)}$.

On the other hand,
\[\Delta_{U(G^{(\lambda)}_{R})}(T_{X})=T_{X}\otimes T_{1}+(T_{1}+\lambda T_{X})\otimes T_{X},\]
\[\Delta_{U(G^{(\lambda)}_{R})}(T_{Y})=T_{Y}\otimes T_{1} +(T_{1}+\lambda T_{X})\otimes T_{Y}.\]
where $\Delta_{U(G^{(\lambda)}_{R})}$ is the comultiplication map of the coordinate ring of  $U(G^{(\lambda)}_{R})$.

Therefore $\tilde{\chi}^{\#}$ is an $R$-Hopf algebra homomorphism. Moreover, 
\[e^{\#}(X)=i^{\#} \circ \tilde{\chi}^{\#}(X), \ \ \ e^{\#}(Y)=i^{\#} \circ \tilde{\chi}^{\#}(Y).\]
which implies the commutativity of the first diagram.

(2) A homomorphism of $R$-Hopf algebras $\tilde{\sigma}^{\#}:R[T_{X^{r_{1}}Y^{r_{2}}}, 1/\Delta]_{1 \leq r_{1}, r_{2} \leq p-1} \rightarrow R[X,Y, 1/1+\lambda X]$ defined by
\[T_{X^{r_{1}}Y^{r_{2}}} \mapsto X^{r_{1}}Y^{r_{2}} \]
is well-defined because $\tilde{\sigma}^{\#}(\Delta)=\displaystyle\prod_{r=0}^{p-1}(1+\lambda X)^{r}$ is invertible in $R[X,Y, 1+\lambda X]$ and $\tilde{\sigma}^{\#}$ is coalgebra homomorphism.
Moreover,
\[i^{\#}(T_{X^{r_{1}}Y^{r_{2}}})=e^{\#} \circ \tilde{\chi}^{\#}(T_{X^{r_{1}}Y^{r_{2}}}).\]
which implies the commutativity of the second diagram.
\end{proof}

\vspace{3mm}
\begin{corollary} Let $S$ be an $R$-scheme and $X$ a $G^{(\lambda)}_{R}$-torsor over $S$. Then the $G^{(\lambda)}_{R}$-torsor $X$ is cleft if and only if the class $[X]$ lies in $\Ker[H^{1}(S,G^{(\lambda)}_{R}) \rightarrow H^{1}(S,{\mathcal H}^{(\lambda)}_{R})]$.
\end{corollary}

\vspace{3mm}
\begin{proof} By Theorem 3.3, (1), we obtain a commutative diagram of pointed sets
\[\begin{CD}
 H^{1}(S,G^{(\lambda)}_{R}) @>{i}>> H^{1}(S,U(G^{(\lambda)}_{R})) \\
 @\vert  @VV{\tilde{\chi}}V \\
 H^{1}(S,G^{(\lambda)}_{R}) @>>{e}> H^{1}(S,{\mathcal H}^{(\lambda)}_{R})
\end{CD}\]
(cf. Demazure-Gabriel [1,Ch III, Prop.4.6]).
From this diagram we deduce that \[\Ker[H^{1}(S,G^{(\lambda)}_{R})\rightarrow  H^{1}(S,U(G^{(\lambda)}_{R}))]\subset \Ker[H^{1}(S,G^{(\lambda)}_{R})\rightarrow H^{1}(S,{\mathcal H}^{(\lambda)}_{R})].\]
On the other hand, by Theorem 3.3, (2), we obtain another commutative diagram of pointed sets
\[\begin{CD}
 H^{1}(S,G^{(\lambda)}_{R}) @>{i}>> H^{1}(S,U(G^{(\lambda)}_{R})) \\
 @\vert  @AA{\tilde{\sigma}}A \\
 H^{1}(S,G^{(\lambda)}_{R}) @>>{e}> H^{1}(S,{\mathcal H}^{(\lambda)}_{R}).
\end{CD}\] which yields the reverse inclusion.
 \[ \Ker[H^{1}(S,G^{(\lambda)}_{R})\rightarrow H^{1}(S,{\mathcal H}^{(\lambda)}_{R})]\subset \Ker[H^{1}(S,G^{(\lambda)}_{R})\rightarrow  H^{1}(S,U(G^{(\lambda)}_{R}))]\].
Combining the two inclusions gives the desired equivalence.
\end{proof}
\vspace{3mm}

\begin{corollary} Let $R$ be an $\F_{p}$-algebra and $\lambda \in R$. Put $G^{(\lambda)}_{R}=\Ker[F:{\mathcal {H}}^{(\lambda)}_{R} \rightarrow {\mathcal {H}}^{(\lambda^{p})}_{R}]$ and set $S=\Spec\,R$. Then a $G^{(\lambda)}_{R}$-torsor $X$ over $S$ is cleft if and only if there exist morphisms 
$X \rightarrow {\mathcal H}^{(\lambda)}_{R}$ and $S \rightarrow {\mathcal {H}}^{(\lambda^{p})}_{R}$ such that the diagram 
\[\begin{CD}
 X @>>> {\mathcal H}^{(\lambda)}_{R} \\
 @VVV  @VV{F}V \\
 S @>>> {\mathcal H}^{(\lambda^{p})}_{R}
\end{CD}\]
is cartesian.
\end{corollary}
\vspace{3mm}
\begin{corollary} Under the notation of Corollary 3.5, the following conditions are equivalent:

(a) Every $G^{(\lambda)}_{R}$-torsor over $R$ is cleft.

(b) The map ${\mathcal{H}}^{(\lambda^{p})}_{R}(R) \rightarrow H^{1}(R, G^{(\lambda)}_{R})$ induced by the exact sequence
\[0 \longrightarrow G^{(\lambda)}_{R} \longrightarrow {\mathcal{H}}^{(\lambda)}_{R} \longrightarrow {\mathcal {H}}^{(\lambda^{p})}_{R} \longrightarrow 0\]
is surjective.

(c) The map $H^{1}(R, {\mathcal H}^{(\lambda)})\rightarrow H^{1}(R, {\mathcal H}^{(\lambda^{p})})$ induced by the Frobenius map $F: {\mathcal H}^{(\lambda)} \rightarrow {\mathcal H}^{(\lambda^{p})}$
is injective.
\end{corollary}
\vspace{3mm}
\begin{rem} From the exact sequence of group schemes
\[0 \longrightarrow \G_{a,R} \overset{i}\longrightarrow {\mathcal H}^{(\lambda)}_{R} \overset{epi}\longrightarrow {\mathcal G}^{(\lambda)}_{R}\longrightarrow 0,\]
we obtain a long exact sequence of pointed sets

\[0 \longrightarrow \G_{a,R}(R) \longrightarrow {\mathcal {H}}^{(\lambda)}_{R}(R) \longrightarrow {\mathcal {G}}^{(\lambda)}_{R}(R) \longrightarrow H^{1}(R, \G_{a,R}) \longrightarrow H^{1}(R, {\mathcal H}^{(\lambda)}_{R}) \longrightarrow H^{1}(R, {\mathcal G}^{(\lambda)}_{R}).\]

If $R$ is a local ring or if $\lambda$ is nilpotent, then $H^{1}(R, {\mathcal H}^{(\lambda)}_{R})=0$ since both $H^{1}(R, \G_{a,R})=0$ and $H^{1}(R, {\mathcal G}^{(\lambda)}_{R})=0$ (\cite{SS}, Cor 1.3). It follows that every $G^{(\lambda)}_{R}$-torsor over $R$ is cleft.  
\end{rem}

\vspace{3mm} By Corollary 3.4, however we can show that non-cleft $G^{(\lambda)}_{R}$-torsor do exist. 

\vspace{3mm}
\begin{example} There exists an $\F_{p}$-algebra $R$ and $\lambda \in R$ such that the map $H^{1}(R, {\mathcal{H}}^{(\lambda)}_{R}) \rightarrow H^{1}(R, {\mathcal{ H}}^{(\lambda^{p})}_{R})$  induced by the Frobenius morphism, is not injective.
Indeed,  consider the commutative diagram with exact rows:

\[\begin{CD}
 0 @>>> \G_{a,R} @>i>> {\mathcal H}^{(\lambda)}_{R}  @>>> {\mathcal G}^{(\lambda)} @>>> 0\\
 @.     @VVFV     @VVFV    @VVFV  \\
 0 @>>> \G_{a,R} @>i>> {\mathcal H}^{(\lambda^{p})}_{R} @>>> {\mathcal G}^{(\lambda^{p})} @>>> 0.
\end{CD}\]

which yields a commutative diagram of pointed sets with exact rows:

\[\begin{CD}
  H^{1}(R,\G_{a,R}) @>i>> H^{1}(R,{\mathcal H}^{(\lambda)}_{R})  @>>> H^{1}(R,{\mathcal G}^{(\lambda)}) @>>>  H^{2}(R,\G_{a,R})\\
 @VVFV     @VVFV     @VVFV    @VVFV   \\
 H^{1}(R,\G_{a,R}) @>i>> H^{1}(R,{\mathcal H}^{(\lambda^{p})}_{R}) @>>> H^{1}(R,{\mathcal G}^{(\lambda^{p})}) @>>>  H^{2}(R,\G_{a,R}).
\end{CD}\]

Moreover, since $H^{i}(R,\G_{a,R})=0$ for $i \geq 1$, it follows that $H^{1}(R,{\mathcal H}^{(\lambda)}_{R})=H^{1}(R,{\mathcal G}^{(\lambda)}_{R})$.
Now assume 
\[R=\F_{p}[X, Y, \frac{1}{Y^{p}+(X+1)^{p}Y+X^{p}}]\]
and $\lambda=X+1$. Then the map $H^{1}(R, {\mathcal {G}}^{(\lambda)}) \rightarrow H^{1}(R, {\mathcal {G}}^{(\lambda^{p})})$ is not injective ([15], Example 4.6).
Consequently, the induced map $H^{1}(R, {\mathcal {H}}^{(\lambda)}_{R}) \rightarrow H^{1}(R, {\mathcal {H}}^{(\lambda^{p})}_{R})$ is also not injective.
\end{example}
\vspace{3mm}

\begin{notation} Let $R$ be a commutative ring and $\lambda \in R$, put $D_{(\lambda, a)}=R[U,V,1/a+\lambda U]$ for $a \in R$. $H^{(\lambda)}_{R}$ denotes the coordinate ring of ${\mathcal{ H}}^{(\lambda)}_{R}$. 
Then
$D_{(\lambda,a)}$ is a right ${H}^{(\lambda)}_{R}$-comodule algebra with structured map 
${H}^{(\lambda)}_{R}$- 
\[\rho:R[U,V,\frac{1}{a+\lambda U}] \rightarrow R[U,V,\frac{1}{a+\lambda U}]\otimes R[X,Y,\frac{1}{1+\lambda X}]:\] defined by
\[U \mapsto U\otimes 1 +(a+\lambda U)\otimes X, \ \ V \mapsto V \otimes 1 +(a+\lambda U)\otimes Y.\]
We put $X_{(\lambda,a)}=\Spec\,D_{(\lambda,a)}$. 
Furthermore, assume that $R$ is an $\F_{p}$-algebra. We put $\tilde{D}_{(\lambda,a,c,d)}=R[U,V]/(X^{p}-c, Y^{p}-d)$ for $c,d\in R$. 
$A^{(\lambda)}_{R}$ denotes the coordinate ring of $G^{(\lambda)}_{R}$. 
Then $\tilde{D}_{(\lambda,a,c,d)}$ is a right $A^{(\lambda)}_{R}$-comodule algebra defined by
a right $A^{(\lambda)}_{R}$-comodule structure map
\[\rho:R[U,V,\frac{1}{a+\lambda U}] \rightarrow R[U,V,\frac{1}{a+\lambda U}]\otimes R[X,Y,\frac{1}{1+\lambda X}]:\]
\[U \mapsto U\otimes 1 +(a+\lambda U)\otimes X, \ \ V \mapsto V \otimes 1 +(a+\lambda U)\otimes Y.\]
We put $\tilde{X}_{(\lambda,a,c,d)}=\Spec\,\tilde{D}_{(\lambda,a,c,d)}$.
\end{notation}
\vspace{3mm}

\begin{prop} Under the above notation, the following assertions hold.\\
(1) If $a$ is invertible in $R/(\lambda)$, then $X_{(\lambda,a)}$ is an ${\mathcal H}^{(\lambda)}_{R}$-torsor over $R$.\\
(2) Assume that $R$ is an $\F_{p}$-algebra. If $a^{p}+\lambda^{p}c$ is invertible in $R$, then $\tilde{X}_{(\lambda,a,c,d)}$ is a $G^{(\lambda)}_{R}$-torsor over $R$. Moreover the contracted product $\tilde{X}_{(\lambda,a,c,d)} \vee^{G^{(\lambda)}_{R}} {\mathcal {H}}^{(\lambda)}_{R}$ is isomorphic to $X_{(\lambda,a)}$ as a right ${\mathcal{H}}^{(\lambda)}_{R}$-torsor.  
\end{prop}
\vspace{3mm}
\begin{proof} (1) Consider the $R$-algebra homomorphism $r: D_{(\lambda, a)}\otimes_{R}D_{(\lambda, a)} \rightarrow D_{(\lambda, a)} \otimes_{R} B_{a}$ defined by
\[U\otimes 1 \mapsto U \otimes 1,\ \ \ V \otimes 1 \mapsto V \otimes 1,\]
\[1\otimes U \mapsto U\otimes 1+(1+\lambda U)\otimes X, \ \ \ 1 \otimes V \mapsto V \otimes 1+(1+\lambda U) \otimes Y.\]
is bijective. The inverse of $r$ is given by 
\[U\otimes 1 \mapsto U \otimes 1,\ \ \ V \otimes 1 \mapsto V \otimes 1,\]
\[1\otimes X \mapsto \frac{-U\otimes 1+1 \otimes U}{(a+\lambda U) \otimes 1}, \ \ \ 1 \otimes Y \mapsto \frac{-V \otimes 1+1\otimes V}{(a+\lambda U)\otimes 1}.\]

Hence, it remains to prove that $D_{(\lambda,a)}$ is faithfully flat over $R$. First note that $D_{(\lambda,a)}$ is flat over $R$ since $D_{(\lambda,a)}$ is a fraction ring of the polynomial ring $R[U,V]$. Next, observe that $D_{(\lambda,a)}\otimes_{R} R/(\lambda)=R[U,V, 1/a+\lambda U]\otimes_{R} R/(\lambda)
=R/(\lambda)[U,V]$ because $a$ is invertible in $R/(\lambda).$ On the other hand, $D_{(\lambda, a)} \otimes_{R} R[1/\lambda]$ is isomorphic, as an $R$-algebra $R[1/\lambda][X, Y, 1/X]$. Therefore, $D_{(\lambda,a)}$ is faithfully flat over $R$. 

(2) We first remark that $a+\lambda U$ is invertible in $\tilde{D}_{(\lambda,a,c,d)}$ since $(a+ \lambda U)^{p}=a^{p}+\lambda^{p}c$ is
invertible in $R$. Therefore the $R$-algebra homomorphism 
\[\tilde{r}: \tilde{D}_{(\lambda,a,c,d)} \otimes \tilde{D}_{(\lambda,a,c,d)} \rightarrow \tilde{D}_{(\lambda,a,c,d)} \otimes A^{(\lambda)}_{R}\]   
defined by 
\[U \otimes 1 \mapsto U \otimes 1 \ \ \ V \otimes 1 \mapsto V \otimes 1\]
\[1 \otimes U \mapsto U \otimes 1+(a+\lambda U)\otimes X \ \ \ 1\otimes V \mapsto V \otimes 1 +(a+\lambda U) \otimes Y\]
is bijective. Indeed, its inverse  is given by
\[U \otimes 1 \mapsto U \otimes 1, \ \ \ V\otimes 1 \mapsto V \otimes 1\]
\[1 \otimes \frac{1\otimes U- U \otimes 1}{(a+ \lambda U)\otimes 1}, \ \ \ 1 \otimes Y \mapsto \frac{1 \otimes V-V \otimes 1}{(a+\lambda U)\otimes Y}.\]
Moreover $\tilde{D}_{(\lambda,a,c,d)}$ is a free $R$-module. Therefore $\tilde{D}_{(\lambda,a,c,d)}$ is faithfully flat over $R$. Define an $R$-algebra homomorphism 
\[\phi: D_{(\lambda,a)} \rightarrow D_{(\lambda,a,c,d)} \otimes_{R} H^{(\lambda)}_{R}\]
by
\[\phi(U)=U\otimes 1+(a+\lambda U)\otimes X, \ \ \ \phi(V)=V \otimes 1 + (a+\lambda U) \otimes Y.\]
This is well-defined since $\phi(a+\lambda U)=(a+ \lambda U)\otimes (1+\lambda X) \in (D_{(\lambda,a,c,d)}\otimes_{R} H^{(\lambda)}_{R})^{\times}$. Moreover $\phi: D_{(\lambda,a)} \rightarrow D_{(\lambda,a,c,d)} \otimes_{R} H^{(\lambda)}_{R}$ is a homomorphism of right $H^{(\lambda)}_{R}$-comodule algebras. 
Now consider the cotensor product $D_{(\lambda,a,c,d)} \square_{A^{(\lambda)}_{R}} H^{(\lambda)}_{R}$ denotes the subalgebra
\[\left\{\displaystyle\sum_{i}a_{i}\otimes b_{i} \in D_{(\lambda,a,c,d)} \otimes H^{(\lambda)}_{R} ; \displaystyle\sum_{i} \rho(a_{i})\otimes b_{i}= \displaystyle\sum_{i}a_{i} \otimes \varDelta(b_{i})\right\},\]
where $\rho$ is the right $A^{(\lambda)}_{R}$-comodule structure map of $D_{(\lambda,a,c,d)}$ and $\varDelta$ is the left $A^{(\lambda)}_{R}$-comodule structure map of $H^{(\lambda)}_{R}$. This is a right $H^{(\lambda)}_{R}$-comodule algebra. 
We then obtain that ${\Im} \phi \subset D_{(\lambda,a,c,d)} \square_{A^{(\lambda)}_{R}} H^{(\lambda)}_{R}$. Finally, note that the right ${\mathcal {H}}^{(\lambda)}$-torsor $\Spec\,D_{(\lambda,a,c,d)} \square_{A^{(\lambda)}_{R}} H^{(\lambda)}_{R}$  coincides with the contracted product $\tilde{X}_{(\lambda,a,c,d)} \vee^{G^{(\lambda)}_{R}} {\mathcal {H}}^{(\lambda)}_{R}$.
Therefore, $\tilde{X}_{(\lambda,a,c,d)} \vee^{G^{(\lambda)}_{R}} {\mathcal H}^{(\lambda)}_{R}$ is isomorphic to $X_{(\lambda,a)}$ as right ${\mathcal {H}}^{(\lambda)}_{R}$-torsor.
\end{proof}
\vspace{3mm}

\begin{corollary} Let $R$ be an $\F_{p}$-algebra and let $\lambda,a,c,d \in R$. Assume that $a^{p}+\lambda^{p}c$ is invertible in $R$. Then the following conditions are equivalent.\\
(a) The $G^{(\lambda)}_{R}$-torsor   $\tilde{X}_{(\lambda,a,c,d)}$ is cleft.\\
(b) The ${\mathcal H}^{(\lambda)}_{R}$-torsor $X_{(\lambda,a)}$ is trivial.\\
(c) There exists $b \in R$ such that $a+\lambda b$ is invertible in $R$.
\end{corollary}










\vspace{5mm}
\section{A calculation of the quotient $U(G^{(\lambda)}_{R}) / G^{(\lambda)}_{R}$ }\label{sec:quotient}

To calculate the quotient $U(G^{(\lambda)}_{R}) / G^{(\lambda)}_{R}$, we recall the results of Doi and Takeuchi. Let $R$ be a commutative ring, $H$ an $R$-Hopf algebra, and $A$ a right $H$-comodule $R$-algebra with $H$-comodule algebra structure $\rho: A \rightarrow A \otimes H$).
We denote by $A^{\mathrm{co}H}=\{a \in A | \rho(a)=a \otimes 1\}$. The coinvariant subalgebra of $A$ under the right $H$-coaction is defined as $A$.
If there exists an $R$-linear map $\phi: H \rightarrow A$ that is also a homomorphism of right $H$-comodules and invertible with respect to the convolution product, then the extension $A/A^{\mathrm{co}H}$ of $R$-algebras is called a cleft extension ([12]). Moreover, the map $\phi$ is called a cleaving map of $A$. 
In this situation, the left $A^{\mathrm{co}H}$-module homomorphism $\Phi : A^{\mathrm{co}H} \otimes H \rightarrow A$ defined by $b \otimes h \mapsto b\phi(h) $ is bijective. Furthermore, define an $R$-linear map $P : A \rightarrow A$ by $a \mapsto \sum a_{(0)}\phi^{-1}(a_{(1)})$ for $a \in A.$ Then,$P(A) \subset A^{\mathrm{co}H}$. Furthermore, \[\Phi^{-1}(a)=\sum P(a_{(0)})\otimes a_{(1)}  \ \ \ \ \ \ \ \ \ \   (*)\] (cf. [2, Theorem 9]).

\vspace{3mm}
The following result is important for the calculation of the quotient $U(G^{(\lambda)}_{R}) / G^{(\lambda)}_{R}$.
\vspace{3mm}

Let $S$ be a scheme and $\varGamma$ an affine commutative group $S$-scheme such that $\mathcal{O}_{\varGamma}$ is a locally free $\mathcal{O}_{S}$-module of finite rank. Then, via the natural closed immersion $i: \varGamma \rightarrow U(\varGamma)$, $U(\varGamma)$ is a cleft $\varGamma$-torsor over $U(\varGamma)/\varGamma$ (see [11, Proposition 3.1]).

\vspace{3mm}

If $S= \Spec \ R$ and $\varGamma= \Spec \ H$, where $H$ is a free $R$-module of finite rank, then, by choosing a basis$\{e_{1}, \dots, e_{n}\}$ of $H$ over $R$, the canonical closed immersion $i: \varGamma \rightarrow U(\varGamma)$ is determined by

\[i^{\#} : R[T_{e_{1}}, \dots, T_{e_{n}}, \frac{1}{D}]\rightarrow H,\]
\[T_{e_{i}} \mapsto e_{i} \ \ \text{for} \ \ 1 \leq i \leq n.\]
Through this $R$-algebra homomorphism $i^{\#}$ which defines the canonical closed immersion $i$, the algebra ${S(H)_{\Theta}}$ acquires the structure of a right $H$-comodule algebra. Moreover, the $R$-linear map \[\phi: H \rightarrow  {S(H)_{\Theta}}\]
\[e_{i} \mapsto T_{e_{i}}  \ \ \text{for} \ \ 1 \leq i \leq n.\]
is a cleaving map of ${S(H)_{\Theta}}$.
Furthermore, we see that $U(\varGamma)/\varGamma$ is isomorphic to $\Spec \ {S(H)_{\Theta}}^{\mathrm{co}H}$ as $R$-schemes since $H$ is finite over $R$.  (cf. [3, Part I, 5.6]) .

\vspace{3mm}

In this section, $R$ denotes an $\F_{p}$-algebra.

\vspace{3mm}

We now consider $H^{(\lambda)}_R$. The canonical closed immersion $i: H^{(\lambda)}_R \rightarrow U(H^{(\lambda)}_R)$ is determined by the

\[ i^{\#}: R[T_{X^{r_{1}}Y^{r_{2}}}, \frac{1}{\varDelta}  ]_{0 \leq r_{1} \leq p-1, 0 \leq r_{2} \leq p-1} \rightarrow R[X,Y]/(X^{p}, Y^{p}),\]
\[ T_{X^{r_{1}}Y^{r_{2}}}  \mapsto  X^{r_{1}}Y^{r_{2}}   \  \text{for} \  0 \leq r_{1} \leq p-1, 0 \leq r_{2} \leq p-1. \]

$R$-Hopf algebras homomorphism $i^{\#}$ which defines the canonical closed immersion $j$, the algebra $S(A_{R}^{(\lambda)})_{\Theta}$ acquires the structure of a right $A_{R}^{(\lambda)}$-comodule algebra. 
Moreover, the $R$-linear map \[\phi: R[X,Y]/(X^{p}, Y^{p}) \rightarrow  R[T_{X^{r_{1}}Y^{r_{2}}}, \frac{1}{\varDelta}  ]_{0 \leq r_{1} \leq p-1, 0 \leq r_{2} \leq p-1},\]
\[X^{r_{1}}Y^{r_{2}} \mapsto  T_{X^{r_{1}}Y^{r_{2}}} \]
is a right $A_{R}^{(\lambda)}$-comodule homomorphism and is invertible with respect to the convolution product.

Therefore, the map is a homomorphism of left $S(A_{R}^{(\lambda)})_{\Theta}^{\mathrm{co}A_{R}^{(\lambda)}}$-modules and right $A_{R}^{(\lambda)}$-comodule,
\[\Phi: S(A_{R}^{(\lambda)})_{\Theta}^{\mathrm{co}A_{R}^{(\lambda)}} \otimes A_{R}^{(\lambda)} \rightarrow S(A_{R}^{(\lambda)})_{\Theta}, \]
\[b \otimes a \mapsto b\phi(a)\]
and it is bijective.

\vspace{3mm}

\begin{notation} We define the $R$-linear map $P_{(\lambda)}: S(A_{R}^{(\lambda)})_{\Theta} \rightarrow S(A_{R}^{(\lambda)})_{\Theta}$
defined by
\[a \mapsto  \sum a_{(0)}\phi^{-1}(a_{(1)})\]
\end{notation}

\vspace{3mm}

\begin{prop} The algebra ${S(A_{R}^{(\lambda)})_{\Theta}}^{\mathrm{co}A_{R}^{(\lambda)}}$ is the subalgebra of  $S(A_{R}^{(\lambda)})_{\Theta}^{\mathrm{co}A_{R}^{(\lambda)}}$ generated by the elements

\[T_{1}, T_{X}^{p}, T_{Y}^{p}, P_{(\lambda)}(T_{X}^{2}), \dots, P_{(\lambda)}(T_{X}^{p-1}),\]
\[P_{(\lambda)}(T_{X}T_{Y}), P_{(\lambda)}(T_{X}^{2}T_{Y}),\dots, P_{(\lambda)}(T_{X}^{p-1}T_{Y}),\]
\[P_{(\lambda)}(T_{X}T_{Y}^{2}), P_{(\lambda)}(T_{X}^{2}T_{Y}^{2}),\dots, P_{(\lambda)}(T_{X}^{p-1}T_{Y}^{2}),\]
\[\vdots\]
\[P_{(\lambda)}(T_{X}T_{Y}^{p-1}), P_{(\lambda)}(T_{X}^{2}T_{Y}^{p-1}),\dots, P_{(\lambda)}(T_{X}^{p-1}T_{Y}^{p-1}),\]
\[T_{1}^{-1}, \frac{1}{T_{1}^{p}+\lambda T_{X}^{p}}, \frac{(\displaystyle\sum_{k=0}^{s} {\binom{s}{k}}\lambda^{k}T_{X^{k}})}{(T_{1}+\lambda T_{X})^{s}}, 2 \le s \le p-1. \]
\end{prop}
\vspace{3mm}

\begin{proof} Let $B$ denote the subalgebra of $S(A_{R}^{(\lambda)})_{\Theta}$ generated by the elements specified above.  First, we clearly have $B \subset S(A_{R}^{(\lambda)})_{\Theta}^{\mathrm{co}A_{R}^{(\lambda)}}$.
Moreover, for $a_{k} \in B$ for $0 \leq k \leq p-1$ such that
\[\Phi^{-1}\Bigr(\frac{T_{1}^{s_{0}}T_{X}^{s_{1}}T_{X^{2}}^{s_{2}}\dots T_{X^{p-1}}^{s_{p-1}} } {(T_{1}+\lambda T_{X})^{t_{1}}\displaystyle\prod_{l=2}^{p-1}\Bigr (\displaystyle\sum_{k'=0}^{l} {\binom{l}{k'}}\lambda^{k'}T_{X^{k'}}\Bigr)^{t_{l}}}\Bigr)= \displaystyle \sum_{k=0}^{p-1}a_{k} \otimes X^{k} \]for $s_{0} \in \Z, s_{1}, \dots s_{p-1}, t_{1}, \dots t_{p-1} \in \N$.
This result is proved in [16, Proposition 3.5], however, for completeness, we provide the argument again. In particular, there exist $a_{0}, a_{1}, \dots a_{p-1} \in B$ such that
\[\Phi^{-1}(T_{X}^{s})= \displaystyle \sum_{k=0}^{p-1}a_{k} \otimes X^{k} \]
for $1 \leq s \leq p-1$.

Since any nonnegative integer $N$ is represented as $N=pm+r$ for some integer $m$ and some integer $r$ such that $0 \leq r < p$,

\[\Phi^{-1}({T_{X}}^{N})=\Phi^{-1}({T_{X}}^{pm+r})=({({T_{X}}^{p})}^{m} \otimes 1)\Phi^{-1}(T_{X}^{r}).\]

We now record the following useful identities:
\[X_{1}+ \lambda^{2} P_{(\lambda)}(T^{2}_{X})=\frac{(T_{1}+\lambda T_{X})^{2}}{T_{1}+2 \lambda T_{X}+ \lambda^{2} T_{X^{2}}},\]and 
 \[T_{1}^{s-1}+\displaystyle\sum_{k=2}^{s-1}{\binom{s}{k}}\lambda^{k}T_{1}^{s-k}P_{\lambda}(T^{k}_{X})+ \lambda^{s}P_{(\lambda)}(T^{s}_{X})=\frac{(T_{1}+\lambda T_{X})^{s}}{\displaystyle\sum_{k=0}^{s}{\binom{s}{k}}\lambda^{k}T_{X^{k}}}\]
 for $3 \leq s \leq p-1$,

Moreover, 

\[\Phi^{-1}\Big(\frac{1}{T_{1}+ \lambda T_{X}}\Big)=\Phi^{-1}\Big(\frac{({T_{1}+ \lambda T_{X}})^{p-1}}{({T_{1}+ \lambda T_{X}})^{p}}\Big)=\Big\{\frac{1}{({T_{1}+ \lambda T_{X}})^{p}} \otimes 1 \Big\}\Phi^{-1}\Big(({T_{1}+ \lambda T_{X}})^{p-1}\Big)\]
\[=\Bigr(\frac{1}{({T_{1}+ \lambda T_{X}})^{p}} \otimes 1 \Bigr)\Bigr\{\sum_{k=0}^{p-1}{\binom{p-1}{k}}\lambda^{k}(T^{p-k-1}_{1} \otimes 1)\Phi^{-1}(T^{k}_{X})\Bigr\}.\]

Hence, for any nonnegative integer $m$, 
there exist $a_{k} \in B$ for $0 \leq k \leq p-1$ such that
\[\Phi^{-1}\Big(\frac{1}{(T_{1}+\lambda T_{X})^{m}}\Big)=\sum_{k=0}^{p-1}(a_{k} \otimes 1)\Phi^{-1}(T^{k}_{X}).\]

For $2 \leq s \leq p-1$, 
\[\Phi^{-1}\Big(\frac{1}{\displaystyle\sum_{k=0}^{s}{\binom{s}{k}}\lambda^{k}T_{X^{k}}}\Big)=\Phi^{-1}\Big(\frac{(T_{1}+\lambda T_{X})^{s}}{\Big\{\displaystyle\sum_{k=0}^{s}{\binom{s}{k}}\lambda^{k}T_{X^{k}}\Big\}(T_{1}+\lambda T_{X})^{s}}\Big)\]
\[=\Big(T_{1}^{s-1}+\displaystyle\sum_{k=2}^{s-1}{\binom{s}{k}}\lambda^{k}T_{1}^{s-k}P_{\lambda}(T^{k}_{X})+ \lambda^{s}P_{(\lambda)}(T^{s}_{X}) \otimes 1\Big)\Phi^{-1}\Bigr(\frac{1}{(T_{1}+\lambda T_{X})^{s}}\Bigr).\]
Hence, 

for any nonnegative integer $m$, there exist elements $a_{k} \in B$ for $0 \leq k \leq p-1$ such that 
\[\Phi^{-1}\Big(\frac{1}{\Big\{\displaystyle\sum_{k=0}^{s}{\binom{s}{k}}\lambda^{k}T_{X^{k}}\Big\}^{m}}\Big)=\sum_{k=0}^{p-1}(a_{k} \otimes 1)\Phi^{-1}(T^{k}_{X}).\]
We also have 
\[P_{(\lambda)}(T^{2}_{X})=\frac{-T_{1}T_{X^{2}}+Q_{2}}{T_{1}+2\lambda T_{X} + \lambda^{2}T_{X^{2}}},\]
where
\[Q_{2} \in R[T_{1}^{\pm 1}, T_{X}, \frac{1}{T_{1}+\lambda T_{X}}].\]
This is proved in [16, Proposition 3.4].
Since 
\[T_{1}+ \lambda^{2} P_{(\lambda)}(T^{2}_{X})=\frac{(T_{1}+\lambda T_{X})^{2}}{T_{1}+2 \lambda T_{X}+ \lambda^{2} T_{X^{2}}},\] we deduce that
\[T_{X^{2}}=\frac{Q_{2}-(T_{1}+2 \lambda T_{X}) P_{(\lambda)}(T^{2}_{X})}{T_{1}+ \lambda^{2} P_{(\lambda)}(T^{2}_{X})}\]
Therefore, for any nonnegative integer $m$, there exist elements $a_{k} \in B$ for $0 \leq k \leq p-1$
such that
\[\Phi^{-1}(T^{m}_{X^{2}})=\sum_{k=0}^{p-1}(a_{k} \otimes 1)\Phi^{-1}(T^{k}_{X}).\]

Suppose that there there exist $a_{k} \in B$ for $0 \leq k \leq p-1$ such that
\[\Phi^{-1}(T_{X^{s}})=\sum_{k=0}^{p-1}(a_{k} \otimes 1)\Phi^{-1}(T^{k}_{X})\] 
 for $2 \leq s \leq p-2$. Then we have
\[P_{(\lambda)}(T^{s+1}_{X})=\frac{\Big\{-T^{s}_{1}-\displaystyle\sum_{k=2}^{s}{\binom{s+1}{k}\lambda^{k}T_{1}^{s+1-k}P_{(\lambda)}(T^{k}_{X})}\Big\}T_{X^{s+1}}+Q_{s+1}}{\displaystyle\sum_{k=0}^{s+1}{\binom{s+1}{k}}\lambda^{k}T_{X^{k}}},\]
where
\[Q_{s+1} \in R[T_{1}^{\pm 1}, T_{X}, \dots, T_{X^{s}}, \frac{1} {\displaystyle\prod_{l=1}^{s}\Big(\displaystyle\sum_{k=0}^{l}{\binom{l}{k}}\lambda^{k}T_{X^{k}}\Big)}]. \]
This is proved in [16, Proposition 3.4].
Hence, 
\[T^{s}_{1}+\sum_{k=2}^{s}{\binom{s+1}{k}\lambda^{k}T_{1}^{s+1-k}P_{(\lambda)}(T^{k}_{X})} + \lambda^{s+1} P_{(\lambda)}(T^{s+1}_{X})=\frac{(T_{1}+ \lambda T_{X})^{s+1}}{\displaystyle\sum_{k=0}^{s+1}{\binom{s+1}{k}}\lambda^{k}T_{X^{k}}},\]
Which implies
\[T_{X^{s+1}}=\frac{-\Big\{\displaystyle\sum_{k=0}^{s}{\binom{s+1}{k}}\lambda^{k}T_{X^{k}}\Big\}P(T^{s+1}_{X})+Q_{s+1}}{T^{s}_{1}+\displaystyle\sum_{k=2}^{s}{\binom{s+1}{k}\lambda^{k}T_{1}^{s+1-k}P_{(\lambda)}(T^{k}_{X})} + \lambda^{s+1} P_{(\lambda)}(T^{s+1}_{X})}.\]
Therefore, for any nonnegative integer $m$, there exist $a_{k} \in B$ for $0 \leq k \leq p-1$
such that
\[\Phi^{-1}(T^{m}_{X^{s+1}})=\sum_{k=0}^{p-1}(a_{k} \otimes 1)\Phi^{-1}(T^{k}_{X})\]
for any nonnegative integer $m$.

\[\Phi^{-1}({T_{Y}}^{N})=\Phi^{-1}({T_{Y}}^{pm+r})=({({T_{Y}}^{p})}^{m} \otimes 1)\Phi^{-1}(T_{Y}^{r}).\]

Since \[\rho(T_{X}^{r_{1}}T_{Y}^{r_{2}})=\Bigr(T_{X}\otimes 1+(T_{1}+\lambda T_{X}) \otimes X\Bigr)^{r_{1}}\Bigr(T_{Y} \otimes 1 + (T_{1}+ \lambda T_{X}) \otimes Y\Bigr)^{r_{2}}\]
for $0 \leq r_{1} \leq p-1, 0 \leq r_{2} \leq p-1$,
there exist $a_{i,j} \in B$ for $0 \leq i \leq p-1, 0 \leq j \leq p-1$ such that

\[\Phi^{-1}({T_{X}^{r_{1}}}{T_{Y}^{r_{2}})}=\displaystyle\sum_{0 \leq i \leq p-1, 0 \leq j \leq p-1}a_{i,j} \otimes X^{i}Y^{j}.\]

Since \[\varDelta(Y)=Y \otimes 1+(1+ \lambda X) \otimes Y,\]
we obtain \[T_{Y}\phi^{-1}(1) +(T_{1}+ \lambda T_{X})\phi^{-1}(Y)=0.\]
Therefore, 
\[ \phi^{-1} (Y)=-\frac{T_{Y} }{T_{1} (T_{1} + \lambda T_{X} )}.\]
Moreover, since \[\varDelta(X^{r_{1}}Y^{r_{2}})=\Bigr(X \otimes (1+ \lambda X) + 1 \otimes X \Bigr)^{r_{1}}\Bigr( Y \otimes 1 +( 1+ \lambda X) \otimes Y\Bigr)^{r_{2}},\]
for  $0 \leq r_{1} \leq p-1, 1 \leq r_{2} \leq p-1$,
we obtain inductively that
\[\phi^{-1}(X^{r_{1}}Y^{r_{2}})= \frac{T_{X^{r_{1}}Y^{r_{2}}}}{\Bigr(\displaystyle \sum_{k=0}^{r_{1}} \binom{r_{1}}{k}\lambda^{k}T_{X^{k}}\Bigr)\Bigr(\displaystyle \sum_{k=0}^{r_{1}+r_{2} (\text{mod} \ p)} \binom{r_{1}+ r_{2} (\text{mod} \ p)}{k}\lambda^{k}T_{X^{k}}\Bigr) }+Q_{r_{1}r_{2}},\]
where 
\[Q_{r_{1}r_{2}} \in R[T_{X^{l_{1}}Y^{l_{2}}}, \frac{1}{\varDelta}  ]_{0 \leq l_{1} \leq p-1, 0 \leq l_{2} \leq r_{2}-1}.\]

Therefore, if $r_{1} \neq 0$ or $r_{2} \neq 1$, we have 
\[P_{\lambda}(T_{X}^{r_{1}}T_{Y}^{r_{2}})=\frac{(T_{1}+\lambda T_{X})^{r_{1}+r_{2}}T_{X^{r_{1}} Y^{r_{2}}}}{ \displaystyle \sum_{k=0}^{r_{1}+r_{2} (\text{mod} \ p)} \binom{r_{1}+ r_{2} (\text{mod} \ p)}{k}\lambda^{k}T_{X^{k}}}+Q'_{r_{1}r_{2}}, \]
where 
\[Q'_{r_{1}r_{2}} \in R[T_{X^{l_{1}}Y^{l_{2}}}, \frac{1}{\varDelta}  ]_{0 \leq l_{1} \leq p-1, 0 \leq l_{2} \leq r_{2}-1}.\]

Hence, still under the assumption $r_{1} \neq 0$ or $r_{2} \neq 1$, it follows that
\[T_{X^{r_{1}} Y^{r_{2}}}=\frac{\Bigr(P_{\lambda}(T_{X}^{r_{1}}T_{Y}^{r_{2}})-Q'_{r_{1}r_{2}} \Bigr)\Bigr(  \displaystyle \sum_{k=0}^{r_{1}+r_{2} (\text{mod} \ p)} \binom{r_{1}+ r_{2} (\text{mod} \ p)}{k}\lambda^{k}T_{X^{k}}       \Bigr)}{(T_{1}+ \lambda T_{X})^{r_{1}+r_{2}}}\]

Therefore, there exist elements $a_{i,j} \in B$ for $0 \leq i \leq p-1, 0 \leq j \leq p-1$ such that

\[\Phi^{-1}(T_{X^{r_{1}} Y^{r_{2}}}^{m})=\displaystyle\sum_{0 \leq i \leq p-1, 0 \leq j \leq p-1}a_{i,j} \otimes X^{i}Y^{j}\]

It follows that the homomorphism of $R$-modules
\[ {\Phi}': B \otimes A_{R}^{(\lambda)}\rightarrow S(A_{R}^{(\lambda)})_{\Theta}  \] 
defined by\[ b \otimes a \mapsto b \phi (a)\]
is bijective.

Finally, considering the natural injection $s: B \rightarrow S(A_{R}^{(\lambda)})_{\Theta}^{\mathrm{co}A_{R}^{(\lambda)}}$,
we obtain the following commutative diagram. 

\[
  \xymatrix{
     B \otimes A_{R}^{(\lambda)}\ar[r]^{{ \ \ \ \  \Phi}'} \ar[d]_{s \otimes \text{Id}} &  S(A_{R}^{(\lambda)})_{\Theta}  \\
     {S(A_{R}^{(\lambda)})_{\Theta}}^{\mathrm{co}A_{R}^{(\lambda)}}  \otimes A_{R}^{(\lambda)}.  \ar[ur]^{\Phi} & 
  }
\]

Moreover, since $A_{R}^{(\lambda)}$ is faithfully flat over $R$, $s$ is bijective.

\end{proof}

\vspace{3mm}
\begin{prop} $S(A_{R}^{(\lambda)})_{\Theta}$ is isomorphic, as an $R$-algebra, to the polynomial algebra
$R[Z_{X^{r_{1}}Y^{r_{2}}}]_{0 \leq r_{1} \leq p-1, 0 \leq r_{2} \leq p-1}$ localized by the elements

\[\begin{split}
&Z_{1}, Z_{1}+\lambda Z_{X}, Z_{1}+ \lambda^{2}Z_{X^{{2}}}, \\
&Z_{1}^{s-1}+\sum_{k=2}^{s-1}\binom{s}{k}\lambda^{k}Z_{1}^{s-k}Z_{X^{k}}+\lambda^{s}Z_{X^{s}}, 3 \le s \le p-1.
\end{split}\]
\end{prop}
\begin{proof}

Let $A_{\mathcal{W}}$ denote the polynomial algebra
$R[Z_{X^{r_{1}}Y^{r_{2}}}]_{0 \leq r_{1} \leq p-1, 0 \leq r_{2} \leq p-1}$ localized by the same elements listed above.
\[\begin{split}
&Z_{1}, Z_{1}+\lambda Z_{X}, Z_{1}+ \lambda^{2}Z_{X^{{2}}}, \\
&Z_{1}^{s-1}+\sum_{k=2}^{s-1}\binom{s}{k}\lambda^{k}Z_{1}^{s-k}Z_{X^{k}}+\lambda^{s}Z_{X^{s}}, 3 \le s \le p-1.
\end{split}\]

We define an $R$-algebras homomorphism
\[\chi:S(A_{R}^{(\lambda)})_{\Theta} \rightarrow  A_{\mathcal{W}}\]
By setting
\[\chi(T_{1})=Z_{1}, \chi(T_{X})=Z_{X},\]

\[\chi(T_{X^{2}})=\frac{\chi(Q_{2})-(Z_{1}+2 \lambda Z_{X}) Z_{X^{2}}}{Z_{1}+ \lambda^{2} Z_{X^{2}}},\]

\[\chi(T_{X^{s+1}})=\frac{-\Big\{\displaystyle\sum_{k=0}^{s}{\binom{s+1}{k}} \lambda^{k}\chi(T_{X^{k}})\Big\}Z_{X^{s+1}}+\chi(Q_{s+1})}{\chi(T^{s}_{1})+\displaystyle\sum_{k=2}^{s}{\binom{s+1}{k}\lambda^{k} (Z_{1}^{s+1-k}) Z_{X^{k}} + \lambda^{s+1}Z_{X^{s+1} }}} \]
and for $2 \leq s \leq p-2$,
\[\chi(T_{Y})=Z_{Y}\]

$\chi(T_{X^{r_{1}} Y^{r_{2}}})=$\[\frac{\Bigr(Z_{X^{r_{1}}Y^{r_{2}}}-\chi(Q'_{r_{1}r_{2}}) \Bigr)\Bigr(  \displaystyle \sum_{k=0}^{r_{1}+r_{2} (\text{mod} \ p)} \binom{r_{1}+ r_{2} (\text{mod} \ p)}{k}\lambda^{k}\chi(T_{X^{k}})       \Bigr)\Bigr(\displaystyle \sum_{k=0}^{r_{1}} \binom{r_{1}}{k}\lambda^{k}\chi(T_{X^{k}})\Bigr)}{(Z_{1}+ \lambda \chi(T_{X}))^{r_{1}+r_{2}}}\]

For mixed monomials, we define $0 \leq r_{1} \leq p-1, 2 \leq r_{2} \leq p-1$.

$\chi(T_{X^{r_{1}} Y^{r_{2}}})=$
\[\frac{\Bigr(Z_{X^{r_{1}}Y^{r_{2}}}-\chi(Q'_{r_{1}r_{2}}) \Bigr)\Bigr(  \displaystyle \sum_{k=0}^{r_{1}+r_{2} (\text{mod} \ p)} \binom{r_{1}+ r_{2} (\text{mod} \ p)}{k}\lambda^{k}\chi(T_{X^{k}})       \Bigr)\Bigr(\displaystyle \sum_{k=0}^{r_{1}} \binom{r_{1}}{k}\lambda^{k}\chi(T_{X^{k}})\Bigr)}{(Z_{1}+ \lambda \chi(T_{X}))^{r_{1}+r_{2}}}\]

for  $1 \leq r_{1} \leq p-1, r_{2}=1$.
 This is well-defined.

Indeed, we obtain 
\[(Z_{1}+2\lambda Z_{X}+\lambda^{2}\chi(T_{X^{2}}))Z_{X^{2}}=\chi(Q_{2})-Z_{1}\chi(T_{X^{2}})\]
since \[(Z_{1}+\lambda^{2}Z_{X^{2}})\chi(T_{X^{2}})=\chi(Q_{2})-(Z_{1}+2 \lambda Z_{X})Z_{X^{2}}.\]

Multiplying both sides by $\lambda^{2}$, we obtain
\[(Z_{1}+2\lambda Z_{X}+\lambda^{2}\chi(T_{X^{2}}))\lambda^{2}Z_{X^{2}}=(Z_{1}+\lambda Z_{X})^{2}-Z_{1}(Z_{1}+2\lambda Z_{X}+\lambda^{2}\chi(T_{X^{2}})).\]
Therefore,
\[(Z_{1}+2\lambda Z_{X}+\lambda^{2}\chi(T_{X^{2}}))(Z_{1}+\lambda^{2}Z_{X^{2}})=(Z_{1}+\lambda Z_{T})^{2}.\]

Hence, \[\frac{1}{(Z_{1}+2\lambda Z_{X}+\lambda^{2}\chi(T_{X^{2}}))}=\frac{Z_{1}+\lambda^{2}Z_{X^{2}}}{(Z_{1}+\lambda Z_{X})^{2}}.\]

For $2 \leq s \leq p-2$, a similar argument yields
\[\Big\{\sum_{k=0}^{s}{\binom{s+1}{k}} \lambda^{k}\chi(T_{X^{k}})+\lambda^{s+1}\chi(T_{X^{s+1}})\Big\}Z_{X^{s+1}}\]
\[=\chi(Q_{s+1})-\Big\{Z^{s}_{1}+\sum_{k=2}^{s}\binom{s+1}{k}\lambda^{k} (Z_{1}^{s+1-k}) Z_{X^{k}}\Big\}\chi(T_{X^{s+1}})\]
since \[\Big\{Z^{s}_{1}+\sum_{k=2}^{s}{\binom{s+1}{k}\lambda^{k} (Z_{1}^{s+1-k}) Z_{X^{k}} + \lambda^{s+1}Z_{X^{s+1}}}\Big\}Z_{X^{s+1}}\]\[=-\Big\{\sum_{k=0}^{s}{\binom{s+1}{k}} \lambda^{k}\chi(T_{X^{k}})\Big\}Z_{X^{s+1}}+\chi(Q_{s+1}).\]
Moreover, by multiplying both sides by $\lambda^{s+1}$, 
\[\Big\{\sum_{k=0}^{s}{\binom{s+1}{k}} \lambda^{k}\chi(T_{X^{k}})+\lambda^{s+1}\chi(T_{X^{s+1}})\Big\}\lambda^{s+1}Z_{X^{s+1}}\]
\[=(Z_{1}+\lambda Z_{X})^{s+1}-\Big\{Z^{s}_{1}+\sum_{k=2}^{s}\binom{s+1}{k}\lambda^{k} Z_{1}^{s+1-k} Z_{X^{k}}\Big\}\Big\{\sum_{k=0}^{s+1}\binom{s+1}{k}\lambda^{k}\chi(T_{X^{k}})\Big\}\]
Therefore,
\[\Big\{Z^{s}_{1}+\sum_{k=2}^{s}\binom{s+1}{k}\lambda^{k} (Z_{1}^{s+1-k}) Z_{X^{k}}+\lambda^{s+1}Z_{X^{s+1}}\Big\}\Big\{\sum_{k=0}^{s+1}\binom{s+1}{k}\lambda^{k}\chi(T_{X^{k}})\Big\}\]
\[=(Z_{1}+\lambda Z_{X})^{s+1}.\]
Thus,
\[\frac{1}{\displaystyle\sum_{k=0}^{s+1}\binom{s+1}{k}\lambda^{k}\chi(T_{X^{k}})} =\frac{(Z_{1}+\lambda Z_{X})^{s+1}}{Z^{s}_{1}+\displaystyle\sum_{k=2}^{s}\binom{s+1}{k}\lambda^{k} (Z_{1}^{s+1-k}) Z_{X^{k}}+\lambda^{s+1}Z_{X^{s+1}}}\]

Now define an $R$-algebras homomorphism
\[\xi: A_{\mathcal{W}} \rightarrow S(A_{R}^{(\lambda)})_{\Theta} \]
by
\[Z_{1} \mapsto T_{1}, Z_{X} \mapsto T_{X},\]

\[Z_{X^{s}} \mapsto P_{(\lambda)}(T^{s}_{X})\]
for $2 \leq s \leq p-1$,

\[Z_{Y} \mapsto T_{Y}\]

\[Z_{X^{r_{1}}Y^{r_{2}}} \mapsto P_{(\lambda)}(T_{X}^{r_{1}}T_{Y}^{r_{2}})\]

for  $1 \leq r_{1} \leq p-1, r_{2}=1$.

\[Z_{X^{r_{1}}Y^{r_{2}}} \mapsto P_{(\lambda)}(T_{X}^{r_{1}}T_{Y}^{r_{2}})\]

for  $0 \leq r_{1} \leq p-1, 2 \leq r_{2} \leq p-1$.

 This is well-defined.

Then $\xi \circ \chi = \mathrm{Id}$ and  $\chi \circ \xi = \mathrm{Id}$. Therefore, $\chi$ is bijective.

\end{proof}

\vspace{3mm}
\begin{theorem} 
$S(A_{R}^{(\lambda)})_{\Theta}^{\mathrm{co}A_{R}^{(\lambda)}}$ is isomorphic as an $R$-algebra to the polynomial algebra
$R[Z_{X^{r_{1}}Y^{r_{2}}}]_{0 \leq r_{1} \leq p-1, 0 \leq r_{2} \leq p-1}$ localized by the elements

\[\begin{split}
&Z_{1}, Z_{1}^{p}+\lambda^{p} Z_{X}, Z_{1}+ \lambda^{2}Z_{X^{{2}}}, \\
&Z_{1}^{s-1}+\sum_{k=2}^{s-1}\binom{s}{k}\lambda^{k}Z_{1}^{s-k}Z_{X^{k}}+\lambda^{s}Z_{X^{s}}, 3 \le s \le p-1.
\end{split}\]
\end{theorem}

\begin{proof}
Let $A_{\mathcal{V}}$ denote the polynomial algebra
$R[Z_{X^{r_{1}}Y^{r_{2}}}]_{0 \leq r_{1} \leq p-1, 0 \leq r_{2} \leq p-1}$ localized by the elements listed above.

\[\begin{split}
&Z_{1}, Z_{1}^{p}+\lambda^{p} Z_{X}, Z_{1}+ \lambda^{2}Z_{X^{{2}}}, \\
&Z_{1}^{s-1}+\sum_{k=2}^{s-1}\binom{s}{k}\lambda^{k}Z_{1}^{s-k}Z_{X^{k}}+\lambda^{s}Z_{X^{s}}, 3 \le s \le p-1.
\end{split}\]
We define an $R$-algebras homomorphism \[\omega: A_{\mathcal{V}} \rightarrow A_{\mathcal{W}} \]
by setting
\[Z_{1} \mapsto Z_{1}, Z_{X} \mapsto Z^{p}_{X},\]

\[Z_{X^{s}} \mapsto Z_{X^{s}},\]
for $2 \leq s \leq p-1$,

\[Z_{X^{r_{1}}Y^{r_{2}}} \mapsto Z_{X^{r_{1}}Y^{r_{2}}}\]

for $0 \leq r_{1} \leq p-1, 1 \leq r_{2} \leq p-1$.
This map is well-defined. Moreover, $\omega$ is injective. Since $\mathrm{Im} (\xi  \circ \omega)= S(A_{R}^{(\lambda)})_{\Theta}^{\mathrm{co}A_{R}^{(\lambda)}}$, the claim follows, and the proof is complete.

\end{proof}

\vspace{5mm}

\begin{rem}Let $R$ be an $\F_{p}$-algebra and $\lambda \in R$. Define the Frobenius map $F: \mathcal{G}^{(\lambda)}_{R} \rightarrow \mathcal{G}^{(\lambda^{p})}_{R}$ on the coordinate rings 
\[R[X, \frac{1}{1+{\lambda}^{p} X}] \rightarrow R[X, \frac{1}{1+\lambda X}]: X \mapsto X^{p}.\]
Put $\varGamma^{(\lambda)}_{R}=\Ker[F:\mathcal{G}^{(\lambda)}_{R} \rightarrow \mathcal{G}^{({\lambda}^{p})}_{R}]$. Then $\varGamma^{(\lambda)}_{R}$ is a commutative finite flat group scheme. In [14], the sculpture and embedding problems were studied for $\varGamma^{(\lambda)}_{R}$.  In [16], the quotient $U(\varGamma^{(\lambda)}_{R})/\varGamma^{(\lambda)}_{R}$ was calculated. Theorem 4.4 is a generalization of this result.
\end{rem}

\section{Acknowledgment}
The author thanks Editage for the English language editing.

\end{document}